\documentclass[reqno]{amsart}
\usepackage{amsmath}
\usepackage{amscd}
\usepackage{graphics}
\usepackage{latexsym}
\usepackage[all]{xy}
\theoremstyle{plain}
\newtheorem{theorem}{Theorem}

\newtheorem*{thm1}{Theorem A}
\newtheorem*{thm3}{Theorem C}
\newtheorem*{thm2}{Theorem B}
\newtheorem{cor}[theorem]{Corollary}
\newtheorem{lemma}[theorem]{Lemma}
\newtheorem{prop}[theorem]{Proposition}

\theoremstyle{definition}
\newtheorem{defn}[theorem]{Definition}

\newtheorem*{rmk}{Remark}

\theoremstyle{remark}

\newcommand{\la}{\langle}
\newcommand{\ra}{\rangle}
\newcommand{\QQ}{\mathbb{Q}}

\newcommand{\ZZ}{\mathbb{Z}}
\newcommand{\CC}{\mathbb{C}}

\newcommand{\GG}{\mathbb{G}}
\newcommand{\PP}{\mathbb{P}}

\newcommand{\OO}{\mathcal O}

\def\PP{{\mathbb P}}

\begin{document}

\title[Global sections]{Birational geometry of the space of complete quadrics} 

\author[Lozano]{C\'esar Lozano Huerta}
\address{University of Illinois at Chicago\\\newline
Department of Mathematics, Statistics and Computer Science \\
Chicago, IL.}
\email{lozano@math.uic.edu} 

\begin{abstract}
We study the birational geometry of the moduli space of complete $n$-quadrics $X$. We exhibit generators for Eff$(X)$ and Nef$(X)$, the cone of effective divisors and the cone of nef divisors, respectively. As a corollary, we show that $X$ is a Fano variety. Furthermore, we run the Minimal Model Program on $X$ and find a moduli interpretation for the models induced by the generators of the nef cone. In the case of complete quadric surfaces, we describe all the birational models of $X$ induced by the movable cone and find a moduli interpretation for some of these models.
\end{abstract}

\maketitle
\section*{Introduction}
\noindent
In 1864 Chasles \cite{CHAS} introduced the space of complete conics in order to solve a famous enumerative problem. This space parameterizes conics with a marking, and it is defined as follows. Let $C$ and $C^*$ be a smooth conic and its dual conic, respectively. The space of complete conics is the closure of the set of pairs $X=\overline{\{(C,C^*)\}}\subset \PP^5\times \PP^{5*}$. Soon after, Schubert in his seminal work \cite{Schu} introduced the higher dimensional generalization: the space of complete $n$-quadrics, which parametrizes marked quadrics and is the space we will study.

\medskip
The space of complete $(n-1)$-quadrics is a compactification of the family of smooth quadric hypersurfaces in $\PP^n$. In this paper we study the birational geometry of this classical space, which we denote by $X_n$, using the Minimal Model Program (MMP). In other words, we wish to understand the collection of all morphisms from $X_n$ into a target variety which is normal and projective. All such target varieties can be defined in terms of the geometry of $X_n$. Indeed, let $X_n(D):=\mathrm{Proj}(\oplus_{m\ge 0}H^0(X_n, mD))$ be a \textit{model} of $X_n$ induced by the divisor $D$. For example, the space of complete conics has two such models: the projections to $\PP^5$ and $\PP^{5*}$. We provide an explicit description of some models in the next dimensional case, which is the space of complete quadric surfaces (Theorem A). For example, we show there are eight distinct birational models. 

\medskip
Remarkably, models obtained by running the MMP on a moduli space often can be interpreted as moduli spaces themselves. A priori, there is no reason for this to be the case. However, Hassett and Keel first exhibited this phenomenon in the context of the Deligne-Mumford compactification of the moduli space of Riemann surfaces $\overline{\mathcal{M}}_g$ \cite{HAS}, \cite{HHD}, \cite{HHF}. In the same vein, Arcara, Bertram, Coskun and Huizenga \cite{ABCH} showed that the same phenomenon holds for the models of $\displaystyle{\mathbf{Hilb}}^n(\PP^2)$, a compactification of the configuration space of $n$ points on the plane. In this case,  they explicitly demonstrate that the models can be interpreted as moduli spaces of Bridgeland stable objects on $\PP^2$. 

\medskip
The purpose of the present paper is to show that the remarkable phenomenon first found by Hassett and Keel also applies to the space of complete quadrics and its birational models. In other words, the birational models of $X_n$ are compactifications of the family of smooth quadrics and they can all be interpreted as moduli spaces. Theorem A and Theorem C are our main results about this.

\medskip
\begin{thm1}\label{THEO1}
\textit{The cone of effective divisors of the space of complete quadric surfaces has eight Mori chambers. Furthermore, there is a moduli interpretation for some of the birational models induced by such a chamber decomposition.}
\end{thm1}

\medskip
In studying the models of a given algebraic variety, the first question we may ask is, how many are there?
The Mori chamber decomposition of the cone of effective divisors has the information about the number of models a given variety may have. However, it is typically very difficult to compute. In the case of complete quadric surfaces $X_3$, we can exhibit such a decomposition explicitly by analyzing the stable base locus of each effective divisor $D\in \mathrm{Eff}(X_3)$. The relation between the Mori chambers and the stable base locus of a divisor has been studied in detail in \cite{POPA1}, \cite{POPA2}. Hence, Theorem A asserts, in particular, that there are finitely many Mori chambers for the space $X_3$, which implies there is a finite number of models of this space. This is an important finiteness property of a so-called Mori dream space. In general, we can show there is a finite number of Mori chambers by showing that the space of complete $(n-1)$-quadrics $X_n$ is a Fano variety, hence a Mori dream space (Corollary \ref{FANO}) by \cite{BC}. 

\medskip
We can say more about the birational models of $X_3$. Each of them can be interpreted as a moduli space of objects described in terms of classical geometry.
For example, a smooth quadric surface $Q\subset \PP^3$ contains two rulings $i.e.$, two $1$-parameter families of lines. It turns out that the smooth quadric $Q\subset \PP^3$ and the family of lines contained in it determine each other. This implies that we get a rational map from the space of complete $2$-quadrics to the space of (flat) families of lines contained in quadrics. This map is not an isomorphism; it induces a \textit{flip}.

\medskip
The proof of Theorem A will rely on the relation between $\overline{\mathcal{M}}_{0,0}(\GG(1,3),2)$, the Kontsevich moduli space of stable maps, and the space of complete $2$-quadrics. The former space is a two-fold ramified cover of the latter (Lemma \ref{2COVER}). The birational geometry of $\overline{\mathcal{M}}_{0,0}(\GG(1,3),2)$ has been analyzed by Chen and Coskun \cite{CC}.

\medskip
In higher dimensions, we start our study of the birational geometry of the space of complete $(n-1)$-quadrics $X_n$ by exhibiting generators of Eff$(X_n)$ and Nef$(X_n)$, the cone of effective and nef divisors, respectively (Theorem B). Using Theorem A as a guiding example, we describe a moduli structure on $X_n(H_k)$, the models induced by the generators of the nef cone (Theorem C). 

\medskip
Now, let us define the space $X_n$. Originally, the following definition is a Theorem in \cite{VAIN}. We will use this result as a definition  of $X_n$ because it is more suitable for the purposes of the present paper. We will present the historically accurate definition of $X_n$ in Section \ref{SEIS}.

\medskip
Let $\PP^N$, where $N=\binom{n+2}{2}-1$, be the space parametrizing quadric hypersurfaces in $\PP^n$. We can stratify $\PP^N$ by the rank of the quadric hypersurfaces,
$$\Phi_1\subset \cdots \subset \Phi_{n-1}\subset \Phi_n\subset \PP^N $$
where $\Phi_i$ denotes the locus of quadrics of rank at most $i$. The space of complete quadrics is obtained as a sequence of blowups of $\PP^N$ along all the $\Phi_i$'s, for $i\le n-1$. 

\medskip
\begin{defn}[Vainsencher, \cite{VAIN}]\label{1.2} Let $\PP^N=X(0)$, and $X(k)=Bl_{\tilde{\Phi}_k}X(k-1)$, where 
$\tilde{\Phi}_k$ denotes the strict transform of the locus of quadrics of rank at most $k$. 
The space of complete $(n-1)$-quadrics is defined as $X_n = Bl_{\tilde{\Phi}_{n-1}}X(n-2)$. 
\end{defn}

It turns out that the exceptional divisors in the previous definition, which we denote by $E_j$ for $1\le j \le n$, generate the effective cone Eff$(X_n)$. Moreover, points along such $E_j$'s parametrize quadrics which are marked over its singular locus by another quadric. In other words, if $Q\in X_n$ is a generic complete quadric of rank $k$, then $Q=(Q',q)$ where $Q'$ is a quadric hypersurface in $\PP^n$ of rank $k$, and $q$ is a smooth quadric over $Sing(Q')\cong \PP^{n-k}$. In this case, $Q$ is in the boundary divisor $E_k$. The divisor $E_n$ is the strict transform of $\Phi_n\subset \PP^N$, hence, $Q\in E_n$ represents a quadric of rank $n$. 

\medskip\noindent
\textit{Example:} Consider $X_3$, the space of complete $2$-quadrics in $\PP^3$. A quadric of rank one, whose marking consists of a double line with two marked points, can be written $Q=(x_0^2,x_1^2,(ax_2+bx_3)^2)\in E_1$. We will study this space in more detail later.

\medskip
Let us introduce generators for the nef divisors whose mention can already be found in Schubert \cite{Schu}. 

\begin{defn}\label{DEFH}
Let $H_i\subset X_n$ denote the closure in $X_n$ of the subvariety parametrizing smooth quadric hypersurfaces in $\PP^{n}$ which are tangent to a fixed linear subspace of dimension $i-1$.
\end{defn}

\begin{thm2} \label{THEO2}
Let $X_n$ be the space of complete $(n\!-\!1)$-quadrics in $\PP^n$. The cone of effective divisors on $X_n$ is generated by boundary divisors $\displaystyle{\mathrm{Eff}}(X_n)=\langle E_1,\ldots , E_{n}\rangle$. Furthermore, the nef cone is generated by $\displaystyle{\mathrm{Nef}}(X_n)=\langle H_1,\ldots , H_{n}\rangle$.
\end{thm2}

\medskip
This result allows us to see that $X_n$ is a Fano variety. Hence,  $X_n$ is a Mori dream space by \cite{BC}. In particular, $N^1(X_n)\otimes \QQ=\mathrm{Pic}(X_n)\otimes \QQ$.

\medskip
Let us define the space which carries the desired moduli structure of $X_n(H_k)$, the models induced by the generators of the nef cone.

\medskip
The second order Chow variety 
$\displaystyle{\mathbf{Chow}}_2(k-1,X_n)$ parametrizes tangent $(k-1)$-planes to complete $(n-1)$-quadrics. In other words, if $Q\subset \PP(V)$ is a smooth quadric hypersurface, then the tangent $k$-planes to $Q$ are parametrized by the Chow form $CF_Q(k)\subset \GG(k,n)$, which is a divisor of degree two. Thus, $[CF_Q(k)]$ is an element in the linear system $|\mathcal{O}_{\GG}(2)|\subset \PP(S^2(\wedge^kV))$. The association $Q\mapsto [CF_Q(k)]$ induces a birational morphism (\cite{BERTRAM}) $$\rho_k:X_n\rightarrow 
\PP(S^2(\wedge^kV))\ .$$
We define the second order Chow variety $\displaystyle{\mathbf{Chow}}_2(k-1,X_n)$ as the image of $\rho_k$.  

\medskip
\begin{thm3} \label{THEO3}
Let $X_n$ be the space of complete $(n-1)$-quadrics. The birational model $X_n(H_k)$, induced by any generator of the nef cone $\displaystyle{\mathrm{Nef}}(X_n)$, is isomorphic to the normalization of $\displaystyle{\mathbf{Chow}}^{\nu}_2(k-1,X_n)$.
\end{thm3}

\bigskip
The paper is organized as follows. Section \ref{UNO} studies higher dimensional quadrics. It contains the proofs of Theorem B and Theorem C. From Section \ref{DOS} onwards, we focus on the case of surfaces. Section \ref{DOS} studies divisors on the space of complete quadric surfaces $X_3$. Section \ref{TRES} contains the stable base locus decomposition of Eff$(X_3)$. Section \ref{CUATRO} describes some birational models that appear in Theorem A. Section \ref{CINCO} contains the proof of Theorem A. We include a final section containing historical remarks in which we describe how this paper unifies results by J.G. Semple \cite{SEM}, \cite{SEMII} using the MMP. We also state a connection of this work with representation theory and GIT-quotients which we would like to explore in the future.
We work over the field of complex numbers throughout.

\section*{Acknowledgments}
\noindent
I would like to express my gratitude to my advisor Izzet Coskun for his guidance and encouragement over the years. I want to thank Dawei Chen for relevant conversations about this work. Thanks to Kevin Whyte and Artie Prendergast-Smith for useful comments about this project, and Francesco Cavazzani for pointing out the precise reference which relates this work to GIT-quotients. Many thanks to Christopher Gomes for improving the language of this work.

\section{Higher dimensional complete quadrics}\label{UNO}

\noindent
In this section, we prove Theorem B and Theorem C.
We denote by $H_i$, for $1 \le i \le n$, the closure in $X_n$ of the smooth quadrics tangent to a fixed $(i-1)$-plane $\Lambda\subset \PP^n$. We refer the reader to \cite{THROOP} for a comprehensive description of the tangency properties of complete quadrics and linear subspaces.

\begin{proof}[Proof of Theorem B] We make use of the following strategy. Let $\overline{\mathrm{NE}}(X_n)$ be the dual cone of $\mathrm{Nef}(X_n)$; this is the Mori cone of effective curves. 
If the divisors $H_i$ are basepoint-free, then $\langle H_1,\ldots, H_n\rangle \subset \mbox{Nef}(X_n)$. The opposite containment is equivalent to $\langle H_1,\ldots, H_n\rangle^{\vee} \subset \mbox{Nef}(X_n)^{\vee}\cong \overline{\mathrm{NE}}(X_n)$. We show this latter statement holds by exhibiting that the dual curves to $\langle H_1,\ldots , H_n\rangle$ are effective curves.

The divisors $H_i$ are basepoint-free. In other words, given $\Lambda_i$, a linear subspace of dimension $i-1$, and $Q\in X_n$ such that $H_i(\Lambda_i)$ vanishes on $Q$, then we can find a distinct $\Lambda'_i$ such that $H_i(\Lambda'_i)$ does not vanish on $Q$. Indeed, if $Q$ is smooth or $\mbox{dim }Sing(Q)< \mbox{codim} \Lambda_i$, then it is clear. If $\mbox{dim }Sing(Q)\ge \mbox{codim} \Lambda_i$, then  $\Lambda_i$ is tangent to the complete quadric $Q=(Q',q)$ as long as the restriction $\Lambda_i|_{Sing(Q)}$ is tangent to the marking-quadric $q$ \cite{THROOP}. If the marking-quadric $q$ is smooth, then $H_i(\Lambda_i')$ does not vanish on $Q$ if the restriction $\Lambda_i'|_{Sing(Q)}$ is not tangent to $q$. In case the marking-quadric $q$ is singular, we repeat the previous argument for the restriction $\Lambda_i|_{Sing(q)}$. So, inductively, we can find $\Lambda_i'$ such that the complete quadric $Q$ is not tangent to $\Lambda_i'$ and consequently $H_i(\Lambda_i')$ does not vanish on $Q$.
Hence, $H_i$ is basepoint-free and $\langle H_1,\ldots ,H_n\rangle \subset \mbox{Nef}(X_n)$. 

Let us show the opposite containment. Consider the following flag, $$\mathbf{Fl}_{\circ}=\{pt=F_1\subset F_2\subset \cdots \subset F_{n+1}=\PP^n \}\ ,$$ where each $F_i$ stands for a linear subspace of dimension $i-1$ contained in $F_{i+1}$. 
Observe that the most singular complete quadric $Q\in X_n$ can be interpreted as the flag $\mathbf{Fl}_{\circ}$ where the nested marking quadrics all have rank $1$. Hence, by letting the subspace $F_i$ vary inside 
$F_{i+1}$ such that it contains $F_{i-1}$, we get a curve 
$\mathbf{Fl}_i\subset X_n$ for each $1\le i\le n$. Observe that 
$\mathbf{Fl}_j.H_i=\delta_{ij}$. This implies that the curves 
$\langle \mathbf{Fl}_1,\ldots , \mathbf{Fl}_n\rangle$ span the dual cone to $\langle H_1,\ldots , H_n\rangle$. Since the
$\mathbf{Fl}_i$ are effective, the result follows.

Let us now prove the claim about the effective cone. It is clear that $\langle E_1,\ldots E_n \rangle \subset \mbox{Eff}(X_n)$. To show that this is an equality, we consider a general effective divisor $D$ and show that it can be written as a linear combination $D=a_1E_1+\cdots + a_nE_n$, where $a_i\ge 0$ for all $i$.
In order to do that, consider the following curves which sweep out each boundary divisor $E_k$, for $1\le k \le n-1$. Let us denote by $B_k$ the $1$-parameter family of complete quadrics $Q=(Q',q)\in E_k$, such that $Q'$ is fixed and the marking quadric $q\subset \PP^{n-k}\cong Sing(Q')$ varies in a general pencil of dual quadrics. The following intersection numbers hold, \begin{equation}
B_k.E_{k}\le 0 \quad \mbox{and }\quad B_k.E_{k+1}>0\ ,\end{equation}
and zero otherwise.
In fact, the number $B_k.E_{k+1}=n-k+1$, as it is the number of times the marking quadric $q$ becomes singular. On the other hand, observe that $B_k\subset E_{k}$, and that the normal bundle $N_{E_k/X_n}\cong \mathcal{O}_{E_k}(-1)$ for $1\le k \le n-1$. Thus, $B_k.E_{k}=c_1(\mathcal{O}_{E_k}(-1)|_{B_k})=-(n-k)$. 

Let $D=a_1E_1+\cdots + a_nE_n$ be a general effective divisor. Then, it does not contain any of the curves $B_k$. This means that $B_k.D\ge 0$, and by $(1)$, we have that $(n-k)a_k\le (n-k+1)a_{k+1}$ for $1\le k \le n-1$. 
This implies that $$a_1\le \tfrac{n}{n-1}a_2\le \ldots \le na_{n}\ .$$ By intersecting $D$ with the pullback to $X_n$ of a general pencil in $\PP^{N*}$, we get that $0\le a_1$. 
\end{proof}

\medskip
The following corollary tells us that the MMP yields finitely many models when applied to the space $X_n$.
\begin{cor}\label{FANO}
The space $X_n$ of complete $(n-1)$-quadrics is a Fano variety, hence a Mori dream space.
\end{cor}
\begin{proof} By definition, one  can compute the canonical class of $X_n$ recursively as follows. Recall our notation $X(1)\cong Bl_{\Phi_1}(\PP^N)$, $X(2)\cong Bl_{\tilde{\Phi}_2}X(1)$, and so on. Then,
\begin{equation*}\begin{aligned}
K_{X_n}&=K_{X(n-2)}+2E_{n-1},\\[-.2cm]
&\vdots \\[-.2cm]
K_{X(i)}&=K_{X(i-1)}+(\Gamma_i-1) E_{i},\\[-.2cm]
&\vdots \\[-.2cm]
K_{X(1)}&=K_{\PP^N}+(\Gamma_1-1) E_{1},\\
\end{aligned}
\end{equation*}
where $N=\binom{n+2}{2}-1$, and $\Gamma_i$ denotes the codimension of the locus $\Phi_i\subset \PP^N$.
From \cite[~page 38]{DECON}, we have that $E_i=2H_i-H_{i-1}-H_{i+1}$, for $1<i<n$. Then, we can write the canonical class $K_{X_n}$ in terms of the generators of the nef cone, $$K_{X_n}=-2H_1-H_2-\cdots -H_{n-1}-2H_n\ .$$
Hence, $X_n$ is Fano by Theorem B, and a Mori dream space by \cite{BC}.
\end{proof}

\medskip
Since the entries of $\wedge^k Q\in \PP(S^2(\wedge^kV))$ are the $(k\times k)$-minors of $Q$, it follows that the birational morphism $\rho_k:X_n\rightarrow \PP(S^2(\wedge^kV))$ is a bijection over the locus of non-singular quadrics. Indeed, given two non-singular matrices $A$ and $B$, if each of the respective $(k\times k)$-minors of $A$ and $B$ are equal, then $A=\lambda B$ for some non-zero scalar $\lambda$. 

\begin{defn}\label{def3}
Let $Q\subset \PP^n=\PP(V)$ be a smooth quadric hypersurface. The second order Chow form $CF_Q(k)\in \PP H^0((\GG(k-1,n),\mathcal{O}(2))\subset \PP(S^2(\wedge^kV))$ parametrizes tangent $(k-1)$-planes to $Q$.
\end{defn}

\begin{lemma}\label{LEMMA5} Let $Q\subset \PP^n$ be a smooth quadric hypersurface. The second order Chow form $CF_Q(k)=\wedge^k Q$ is equal to the $k$-th wedge of $Q$.
\end{lemma}
\begin{proof}
Let $L\subset \PP^n$ be a $(k-1)$-plane, and $q=Q|_{L}$ be the restriction of $Q$ to $L$. If $L$ is not contained in $Q$, then the $(k-1)$-plane $L$ is tangent to $Q$ if $q$ is singular, which is equivalent to $\mbox{det}\  q=0$. Observe that $L\in \GG(k-1,n)$ belongs to the zero locus of the Chow form $CF_Q(k)$ if and only if $L$ belongs to the zero locus of the quadric $\wedge^kQ$. Indeed, 
\begin{equation*}
\begin{aligned}
L^t (\wedge^kQ) L =&\wedge^k(L^t Q L)\\
=&\wedge^k q \\
=& \mbox{det}( q)\ .
\end{aligned}
\end{equation*}
It follows that $\mbox{det}\ q=0$ if and only if $L$ is in the zero locus of $\wedge^k Q$. Hence, $CF_Q(k)$ and $\wedge^k Q$ define the same divisor on $\GG(k-1,n)$.
\end{proof}

\medskip
Lemma \ref{LEMMA5} implies that the image of the morphism $\rho_k:X_n\rightarrow \PP(S^2(\wedge^kV))$ carries a moduli interpretation: it parametrizes tangent $(k-1)$-planes to complete quadrics. We define the second order Chow variety $\mathbf{Chow}_2(k-1,X_n)$ as the image of $\rho_k(X_n)\subset \PP(S^2(\wedge^kV))$.

\begin{thm3} Let $H_k$ be a generator of $\displaystyle{\mathrm{Nef}}(X_n)$, the nef cone of $X_n$. For each $1 \le k\le n$, the model $X_n(H_k)$ is isomorphic to the normalization of the second order Chow variety, $$X_n(H_k)=\mathrm{Proj}\left(\bigoplus_{m\ge 0}H^0(X_n,mH_k)\right)\cong \mathbf{Chow}_2(k-1,X_n)^{\nu}. $$ 
\end{thm3}
\begin{proof}
In order to establish that $X_n(H_k)\cong \rho_k(X_n)^{\nu}$, it suffices to show that both $\rho_k$ and the induced map $\phi_{H_k}:X_n\rightarrow X_n(H_k)$ contract the same extremal rays in $\overline{\mathrm{NE}}(X_n)$ (See, \cite{LAZ} for a proof of this fact). 
By Theorem B, we know that $\phi_{H_k}$ contracts the classes $\mathbf{Fl}_j$ for $j\ne k$, which generate the Mori cone of curves $\overline{\mathrm{NE}}(X_n)$. In order to show that the morphism $\rho_k$ contracts those same curve classes, we use a parametric representation of them. Let us describe such a parametrization.

The following description follows closely \cite{SEMII} and \cite{TYRELL}. We write a complete quadric as $Q=M^t q M$, where the matrix $M=(M_{ij})$ has $1$'s along the diagonal, and the entries $M_{k,k+1}=t_k$ are affine parameters above the diagonal, and zero otherwise. 
For example, $M$ has the following form in the case $n=3$,
$$M=\left(
\begin{array}{cccc}
1&t_1&0&0\\
&1&t_2 &0\\
&&1&t_3\\
& & &1
\end{array}
\right) .$$ The matrix $q=[1,q_1,q_1q_2,\ldots , q_1\cdots q_n]$ is a diagonal matrix, where $q_j$ are affine parameters. 
Observe that the matrix $M$, as described above, and $q_r=0$, for $1\le r\le n$, give rise to the complete quadric $$Q=M^tqM=((x_0+t_1x_1)^2,(x_1+t_2x_2)^2,\ldots ,(x_{n-1}+t_nx_n)^2),$$ where the marking has rank $1$  (Section \ref{SEIS} further clarifies this notation).

We obtain a parametrization of the representatives for the curve classes $\mathbf{F}_j\in \overline{\mathrm{NE}}(X_n)$, when $q_1=\cdots =q_n=0$ and $t_k=0$ for all $k\ne j$ \cite{TYRELL}.
Hence, the parameter $t_j$ in the expression of $M$, is an affine parameter of the curve $\mathbf{Fl}_j$.

In order to conclude that $\rho_k$ contracts $\mathbf{Fl}_j$ for $j\ne k$, it suffices to show that if $t_l=0$, for $l\ne j$ ($i.e.$, only $t_j$, the parameter of $F_j$, survives), then the form $\wedge^kQ=\wedge^k M^t(\wedge^k q) \wedge^k M$ is constant.
For example, consider $n=3$. From the following matrix,  
$$\wedge^2Q=\left(
\begin{array}{cccccc}
1&t_2&0&t_1t_2&0&0\\
t_2&1&0 &t_1t_2^2&0&0\\
0&0&0&0&0&0\\
t_1t_2&t_1t_2^2 &0 &t_1^2t_2^2&0&0\\
0&0 &0 &0&0&0\\
0&0 &0 &0&0&0\\
\end{array}
\right), $$ 
it follows that $\rho_2$ contracts $\mathbf{F}_1=\{t_2=t_3=0\}$ and $\mathbf{F}_3=\{t_1=t_2=0\}$. The fact that $q=[1,0,\ldots,0]$ simplifies greatly the computation of $\wedge^kQ$ in general. We omit the details since no difficulty arises. This completes the proof.
\end{proof}

\begin{rmk} Following the historically accurate definition of $X_n$ (see, Section \ref{SEIS}), the morphisms $\rho_k$ are very natural projection maps. A good deal of the rest of the paper is devoted to fully understanding all of the maps $\rho_k$ in the case $n=3$.\end{rmk}

\section{Divisors on the space of complete $2$-quadrics}\label{DOS}

\medskip\noindent
J.G. Semple in \cite{SEM}, \cite{SEMII} studied in detail $X_3$, the space of complete quadric surfaces. The rest of the paper is devoted to further studying this space by applying the Minimal Model Program. Indeed, we will interpret the spaces studied by Semple as models $X_3(D)=\mathrm{Proj}(\oplus_{m\ge 0}H^0(X_3, mD))$, where $D$ lies in the cone $\mathrm{Nef}(X_3)$. This section contains the preliminaries needed to show more; we aim to describe all the models $X_3(D)$, where $D$ lies in a larger cone; the movable cone $\mathrm{Mov}(X_3)$ (Definition \ref{MOVABLE}). Moreover, we exhibit some of these models as moduli spaces.

\medskip
Let $\Phi_2$ denote the locus of symmetric $(3\times 3)$-matrices of rank at most two. By definition, the space $X_3\cong Bl_{\tilde{\Phi}_2} X(1)$, where $X(1)$ is a blowup of $\PP^9$ along $\Phi_1$, the locus of symmetric matrices of rank at most one. We can also obtain $X_3$ by blowing up $\PP^{9*}$, the space of dual quadrics in $\PP^3$, in a similar manner. Let us interpret these spaces as models of $X_3$.

\medskip	
The divisor classes $H_i$ in $\mathrm{Pic}(X_3)$, as in Definition (\ref{DEFH}), coincide with the class of the strict transform of a generator of the ideal of $\Phi_i\subset \PP^9$. Indeed, let us denote by $p:X_3 \rightarrow \PP^9$ the blowup map, clearly $p^*(\OO_{\PP^9}(1))=H_1$. Moreover, let $h_1$ and $h_2$ be two generators of the ideals $I(\Phi_1)$ and $I(\Phi_2)$, respectively. Since
\begin{equation*}\begin{aligned}
p^*([h_1])&=2H_1-E_1,\\
p^*([h_2])&	=3H_1-2E_1-E_2,
\end{aligned}\end{equation*} in $\mathrm{Pic}(X_3)$, we can compare these classes with those of $H_2$ and $H_3$.

\begin{lemma}\label{1.7}
Let $H_2, H_3$ be the divisors as in Definition \ref{DEFH}. Their classes in $\mathrm{Pic}(X_3)$ are
\begin{equation}\label{EQ}
\begin{aligned}
H_2&=2H_1-E_1 , \\
H_3&=3H_1-2E_1-E_2 .
\end{aligned}\end{equation}
\end{lemma}
\begin{proof} 
Let $G,C_2,L_2\subset X_3$ be the following test curves. The curve $G$ stands for a general pencil. $C_2$ is defined by the product of a fixed plane $P_0$ and a pencil of planes $P_t$ such that $C_2=\{P_0P_t\}$. The curve $L_2$ is defined by fixing two planes whose intersection is the line $l$ and letting one of the two marked points on $l$ vary. 

The following numbers determine the class of $H_i$ for $i=2,3$.
\begin{equation*}
\begin{aligned}
 G.H_1=1\quad & C_2.H_1=1 & L_2.H_1=0 \\
 G.H_2=2 \quad  & C_2.H_2=0 & L_2.H_2=0\\
G.H_3=3\quad & C_2.H_3=0 & L_2.H_3=1 \\
G.E_1=0\quad & C_2.E_1=2 & L_2.E_1=0 \\
G.E_2=0\quad & C_2.E_2=-1  & L_2.E_2=-1 
\end{aligned}
\end{equation*}
The normal bundle $N_{E_2\backslash {X_3}}\cong \mathcal{O}_{E_2}(-1)$. Then, the restriction to the generic line $L_2\subset E_2$ is isomorphic to $\mathcal{O}_{\PP^1}(-1)$, hence $L_2.E_2=-1$. Similarly $C_2.E_2=-1$.
If we write $H_i=aH_1+bE_2+cE_3$ for $i=2,3$, and use the test curves $G,C_2,L_2$ to find the values of $a,b,c$, the result follows. 
\end{proof}

The following proposition complements Theorem A; see diagram \ref{DIAGRAM}. We will denote by $H_1$ the pull-back of $\mathcal{O}_{\PP^9}(1)$, and $H_2$ the pull-back of a generator of the ideal $I(\Phi_1)$. 

\begin{prop}\label{X1}
Let $X(1)=Bl_{\Phi_1}\PP^9$. Then, $\mathrm{Nef}(X(1))\cong \langle H_1,H_2\rangle$.
\end{prop}
\begin{proof}
$H_1$ is clearly an extremal ray of the nef cone. The divisor $H_2$ is basepoint-free by definition and we have that $H_2.S=0$, where $S$ denotes a pencil $Q_1+tQ_1'$ of quadrics of rank $1$. Since the curve $S$ sweeps out the secant variety $Sec(\Phi_1)=\Phi_2$, then $H_2$ induces a small contraction $\phi_1:X(1)\rightarrow Z$. 
\end{proof}
The canonical divisor $K_{X(1)}=-10H_1-5E_1=-5H_2$. Hence, $K_{X(1)}.S=0$, and $X(1)$ is a flop of a space $Y(1)$ over $Z$ in the following diagram,
\begin{equation}\label{X1DIAGRAM}
\xymatrix @!=1pc{  &  &X_3\ar[dr]^{\pi_2} \ar[dl]_{\pi_1}& & \\
& X(1)\ar[dl] \ar@{.>}[rr]^{\mbox{\tiny{flop}}} \ar[dr]^{\phi_1}&  &  Y(1) \ar[dl]_{\phi_2}\ar[dr]&\\
\PP^{9}&&Z&&  \PP^{9*}} 
\end{equation}
where $X(1)$ as above, and $Y(1)=Bl_{\Phi_1}\PP^{9*}$. The morphism $\phi_1$ is induced by the sub-linear series of $(2\times 2)$-minors cutting out $\Phi_1\subset \PP^9$ scheme-theoretically.
Observe that the morphisms both $\pi_1$ and $\pi_2$ contract the divisor $E_2$.

\medskip
The following divisor class, and the model induced by it, was not analyzed in \cite{SEMII}. This constitutes new information about the birational geometry of $X_3$.

\begin{defn}\label{P} Let $P$ denote the closure in $X_3$ of smooth quadrics $Q$ such that the induced $2$-plane $\Lambda_Q\subset \PP^5$ by one of the rulings of $Q$ has a non-empty intersection with a fixed $2$-plane in the Pl\"{u}cker embedding of $\GG(1,3)$.
\end{defn}
\begin{lemma}\label{DIVISORP}
The divisor class of $P$ is $$[P]=2(2H_1-H_2+2H_3) \ .$$
\end{lemma}

\begin{proof}The assertion follows from the following intersection numbers
$$ P.C_1^*=0, \quad P.R_2=4, \quad
P.C_3=0 ,$$
where the curves $C_1^*, R_2,C_3$ are defined as follows. The curve $C_1^*$ is a double plane with a pencil of dual conics on it, $R_2$ denotes the strict transform to $X_3$ of the pencil $Q_2+\lambda Q'_2$, where $Q_2$ and $Q'_2$ denote quadrics of rank 2 in $\PP^3$, and the curve class $C_3$ is defined by a cone over a general pencil of conics.

Let us compute the intersection number $P.R_2$. Since $R_2.E_2=2$ and $R_2.E_1=R_2.E_3=0$, then it induces a $2$-fold cover of curves $\gamma(\lambda)\rightarrow R_2(\lambda)$, where $\gamma(\lambda)$ represents the curve of 2-planes induced by the pencil $R_2$. 
Indeed, for each $\lambda$, the lines contained in the complete quadric $Q(\lambda)\in R_2(\lambda)$ form two curves $C_{\lambda}, C'_{\lambda} $ in the Grassmannian
$\GG(1,3)$. This is the Fano variety of lines $F_1(Q)$ (or a flat limit of it). Each such curve $C_{\lambda}$ is contained in a unique $2$-plane $\Lambda_{C_{\lambda}}\subset \PP^5$. Consequently, $P.R_2=deg(\gamma)$ as a subvariety of $\GG(2,5)$. On the other hand, the class of the surface $S$ that a curve $C_{\lambda}$ sweeps out in the Grassmannian $\GG(1,3)$ (as we vary $\lambda$), is $[S]=\sigma_2+\sigma_{1,1}\in 
A_2(\GG(1,3))$. Thus,
\begin{equation*}
\begin{aligned}
P.R_2&=\mbox{deg }\gamma & \quad\mbox{in }\GG(2,5)\\
&=2S.\sigma_1^2 \\
&=2(\sigma_{1,1}+\sigma_2)\sigma^2_1 & \quad \mbox{in }\GG(1,3)\\
&=4 
\end{aligned}
\end{equation*}
The numbers $P.C_1^*=0$ and $C_3.P=0$ follow from the fact that all the conics induced by them lie in a fixed plane. The result follows.\end{proof}

\section{Stable Base Locus Decomposition}\label{TRES}

\noindent
Two divisors $D_1,D_2$ induce Mori equivalent models $X(D_1), X(D_2)$ as long as they belong to the same Mori chamber. Thus, we can partition the cone $\mathrm{Eff}(X)$ into Mori chambers by looking at the models $X(D)$. Typically, Mori chambers are very difficult to compute. In order to describe the Mori chamber decomposition of $\mathrm{Eff}(X)$ we use the fact that the Mori chambers can be identified by looking at the stable base locus of the respective divisors. This relation among Mori chambers and the stable base locus decomposition has been studied in \cite{POPA1, POPA2}. In our case, there will be finitely many chambers in $\mathrm{Eff}(X_3)$ as the space of complete quadric surfaces is a Mori dream space.

\medskip
Divisors $D$ for which the map $\phi_{D}:X\rightarrow X(D)$ is an isomorphism in codimension one are called small modifications of $X$, and are of special importance: they give rise to divisorial contractions and flips of $X$. Such divisors are called movable and they form the so-called movable cone $\mathrm{Mov}(X)$ (Definition \ref{MOVABLE}). We will focus on studying models $X_3(D)$, where $D$ is movable.

\medskip
In this section, we compute the base locus decomposition of $\mathrm{Eff}(X_3)$. In order to do that, we need curve classes and their intersection numbers with divisors. We summarize such intersection numbers in the following table, and define the curve classes immediately after.

$$
\begin {tabular}{|c|c|c|c|c|c|c|c|c|}
Curve class & $C.H_1$ & $C.H_2$  & $C.H_3$ & $C.E_1$ & $C.E_2$ & $C.E_3$ & Deformation cover \\ 
\hline  
$G$ & $1$ & $2$ & $3$ & $0$ & $0$ & $4$ & $X_3$    \\[2mm]
$G^*$ & 3 & 2 & 1& $4$ & $0$ & $0$ & $X_3$    \\[2mm]
$C_1$ & $0$ & $1$ & $2$ & $-1$ & $0$ & $3$ & $E_1$     \\[2mm]
$C_1^*$ & $0$ & $2$ & $1$ & $-2$ & $3$ & $0$ & $E_1$     \\[2mm]
$C_2$ & $1$ & $0$ & $0$ & $2$ & $-1$ & $0$ & $E_2$          \\[2mm]
$C_3$ & $1$ & $2$ & $0$ & $0$ & $3$ & $-2$ & $E_3$     \\[2mm]
$C_{1,2}$ & $0$ & $1$ & $0$ & $-1$ & $2$ & $-1$ & $E_1\cap E_3$     \\[2mm]
$L_2$ & $0$ & $0$ & $1$ & $0$ & $-1$ & $2$ & $E_2$          \\[2mm]
 \end {tabular}
$$

\medskip
\noindent 
The curve class $G$ (respectively, $G^*$) stands for the strict transform to $X_3$ of a general pencil in $\PP^9$ (respectively, $\PP^{9*}$). The class of $C_1$ (respectively, $C_1^*$) is defined by considering a general pencil of conics (respectively, dual conics) on a fixed double plane. The curve class $C_2$ is defined as the product of a fixed plane $P_0$ and a pencil of planes $P_t$ such that the marking is fixed. The class $C_3$ consists of the cone over the pencil of conics in a fixed plane. The curve $C_{1,2}$ consists of a pencil of rank 2 conics on a fixed double plane. Such a pencil of conics is a fixed line $l_0$ and a pencil of lines whose base-locus is on $l_0$. Similarly, the curve $L_2$ is defined by fixing two planes whose intersection is the line $l$ and letting one of the two marked points on $l$ vary. 

\textbf{Notation.} We denote by $c(H_1,\overline{P})$ the positive linear combinations $aH_1+bP$ such that $0\le a$ and $0< b$.

\begin{prop}\label{prop2}
The stable base locus decomposition partitions the effective cone of $X_3$ into the following chambers:
\begin{itemize}
\item[(1)] In the closed cone spanned by non-negative linear combinations of $\la H_1, H_2, H_3 \ra$, the stable base locus is empty.
\item[(2)] In the domain spanned by positive linear combinations of $\langle H_1,H_3,P\rangle$ along with the set $c(H_1,\overline{P})\cup c(H_3,\overline{P})$, the stable base locus is $E_1\cap E_3$ and consists of double planes marked with a singular conic of rank 2.
\item[(3)] In the domain spanned by positive linear combinations of $\la H_3,E_3,P\ra $ along with $c(H_3,\overline{E}_3)\cup c(P,\overline{E}_3)$, the stable base locus consists of the divisor $E_3$.
\item[(4)] In the domain spanned by positive linear combinations of $\la H_1, E_1, P\ra$ along with $c(H_1,\overline{E}_1)\cup c(P,\overline{E}_1)$, the stable base locus consists of the divisor $E_1$.
\item[(5)] In the domain spanned by positive linear combinations of $\la P, E_1, E_3\ra $ along with $c(E_1,E_3)$, the stable base locus consists of the union $E_1\cup E_3$.
\item[(6)] In the domain spanned by positive linear combinations of $\la H_3, E_2, E_3\ra $ along with $c(E_2,E_3)$, the stable base locus consists of the union $E_2\cup E_3$.
\item[(7)] In the domain spanned by positive linear combinations of $\la H_1, E_1, E_2\ra$ along with $c(E_1,E_2)$, the stable base locus consists of the union $E_1\cup E_2$.
\item[(8)] In the domain bounded by $H_1,H_2,H_3$ and $E_2$ along with $c(H_1,\overline{E}_2)\cup c(H_3,\overline{E}_2)$, the stable base locus consists of the divisor $E_2$.
\end{itemize}
\end{prop}
\begin{center}
\begin{figure}[htb]
\resizebox{1\textwidth}{!}{\includegraphics{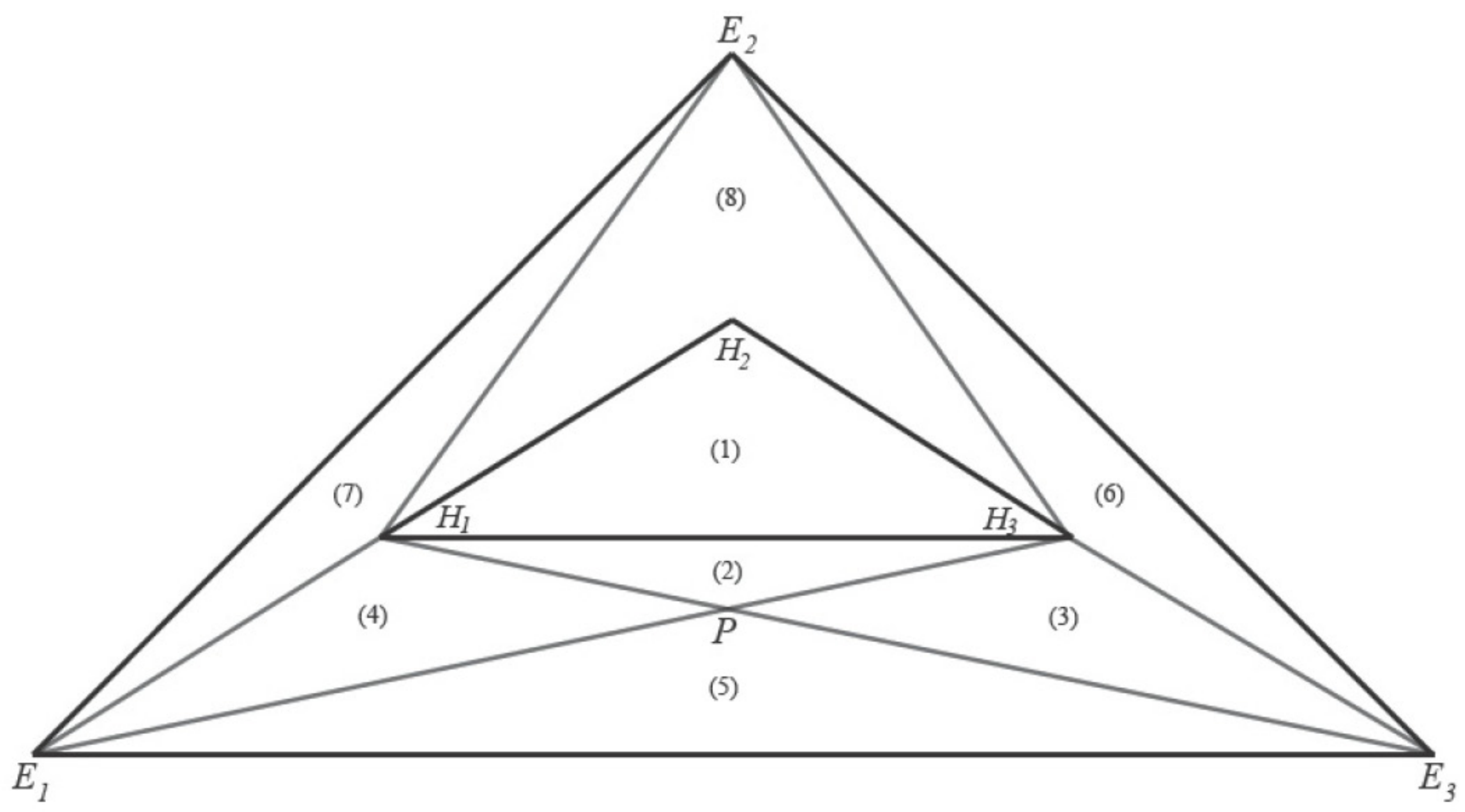}}
\caption{Stable base locus decomposition of $\mathrm{Eff}(X_3)$.}
\end{figure}
\end{center}
\begin{proof} We will make use of the symmetry induced by the map $\xi:X_3 \rightarrow X_3$ defined by sending the quadric $Q$ to $\wedge^3Q$,
$$\xi: Q\longmapsto \wedge^3Q \ .$$ Note that $\xi$ maps $E_1$ (respectively, $H_1$) to $E_3$ (respectively, $H_3$) and keeps $E_2$ (respectively, $H_2$) fixed. The stable base locus of the divisor $\xi^*(D)$ is equal to the inverse image under $\xi$ of the stable base locus of $D$. The symmetry given by $\xi$ will simplify our calculations.

By Theorem B, any divisor contained in the closed cone generated by $H_1$, $H_2$ and $H_3$ is basepoint-free, hence its (stable) base locus is empty.

Let $D$ be a general divisor $D=aH_1+bH_2+cH_3$. Consider the curves $C_1$ and $C_3$ as defined above.
Then,
\begin{equation}\label{3}
C_1.D=b+2c, \qquad C_3.D=a+2b \ .
\end{equation}
Since the curve $C_1$ (respectively, $C_3$) covers $E_1$ (respectively, $E_3$), it follows that $E_1$ (respectively, $E_3$) is in the base locus of any divisor $D$ satisfying $b+2c<0$ (respectively, $a+2b<0$).

On the other hand, $\xi$ maps the plane $b+2c=0$ to the plane $b+2a=0$. Consequently, $E_3$ is in the base locus of any divisor satisfying $b+2a<0$. Similarly, $E_1$ is in the base locus of the linear system $|D|$ if $c+2b<0$.
We conclude that $E_1$ is in the base locus of any divisor contained in the region bounded by $E_1,E_2,H_1$ and $E_3$. Similarly, $E_3$ is in the base locus of any divisor contained in the region bounded by $E_3,E_2,H_3$ and $E_1$.

Let the curve classes $C_2$ and $L_2$ be as defined above. We have that,
\begin{equation}\label{4}
C_2.D=a, \qquad L_2.D=c \ .
\end{equation}
Since both the curves $C_2$ and $L_2$ cover the divisor $E_2$, then $E_2$ is in the base locus of any divisor $D$ satisfying $a<0$ as well as $c<0$. The inequality $a<0$, tells us the union 
$E_2\cup E_3$ is in the base locus of any divisor in the region spanned by positive linear combinations of $\la E_3, H_3, E_2\ra $ along with the set $c(E_3,\overline{E}_2)\cup c(H_3,
\overline{E}_2)$. Similarly, the union $E_2\cup E_1$ is in the base locus of any divisor in the region spanned by positive linear combinations of $\la E_1,E_2, H_1 \ra $ along with the set $c(E_1,\overline{E}_2)\cup c(H_1,\overline{E}_2)$.
By intersecting these two regions, the union $E_3\cup E_1$ is in the base locus of any divisor in the region spanned by positive linear combinations of $\la E_1,P, E_3 \ra $ along with the set
$c(E_1,E_3)$.

By the equation (5) above, $E_3$ is in the base locus of any divisor in the region spanned by positive linear combinations of $E_3,H_3$ and $P$, along with $c(P,\overline{E}_3)\cup c(H_3,\overline{E}_3)$. Symmetry implies that $E_1$ is in the base locus in the region spanned by positive linear combinations of $P$, $H_1$ and $E_1$ along with $c(P,\overline{E}_1)\cup c(H_1,
\overline{E}_1)$. 

Let $C_{1,2}$ be the curve as defined above. We have that 
$$C_{1,2}.D=b \ . $$
Since the deformations of $C_{1,2}$ cover the intersection $E_1\cap E_3$, this locus $E_1\cap E_3$ is in the base locus of any divisor contained in the region bounded by $H_1$, $H_3$ $P$ and $c(H_1,\overline{P})\cup c(H_3,\overline{P})$.
Finally, $E_2$ is in the base locus for any divisor $D$ in the region bounded by $H_1, H_2,H_3$ and $E_2$. This is the description of the base locus decomposition of $\mathrm{Eff}(X_3)$.

In order to finish the proof, we need to show that the stable base locus does not get any bigger than our description of it above.

(i) The divisors $H_1,H_2,H_3$ are basepoint-free by Theorem B. Hence, for divisors contained in the closed cone generated by $H_1,H_2,H_3$ the base locus is empty.

(ii) Since $H_1$ and $H_3$ are basepoint-free, it follows that for any divisor $D$ in the interior of the cone generated by $H_1, H_2$ and $P$, the base locus of the linear system $|D|$ is contained in that of $|P|$. The same applies for the walls $c(H_1,\overline{P})$ and $c(H_3,\overline{P})$. Observe that the base locus of $|P|$ is the locus in $X_3$ parametrizing those complete quadric surfaces whose rulings induce a double line with two 
marked points in $\GG(1,3)$ (Proposition \ref{CC}) . Indeed, for any complete quadric $Q$ inducing either rank $3$ or $2$ conics in $\GG(1,3)$, there is a unique $2$-plane in $\PP^5$ containing them. The indeterminacy of $|P|$ does not get bigger because for any pair of $2$-planes $\Lambda_i$ ($i=1,2$) in $\PP^5$, we can find another $2$-plane missing them both. It follows that for $Q$ a quadric defining a $2$-plane $\Lambda\subset \PP^5$, there is a $D\in |P|$, such that $D$ does not vanish at $Q$. We conclude that the quadrics inducing double lines with two marked points in $\GG(1,3)$ are in the base locus of $P$ $i.e.$, $E_1\cap E_3$.

(iii-iv) Since $P$ can be written as $P=E_3+2H_1=E_1+2H_3$, it follows that for any divisor $D$ contained in the interior of the cone 
generated by $E_3, H_3$ and $P$ or along the wall $c(H_3,\overline{P})$, the base locus of $D$ must be contained in $E_3$. Similarly, for any divisor $D$ contained in the interior of the cone of $E_1, H_1$ and 
$P$ or along the walls $c(H_1,\overline{P})$, the base locus of $D$ must be contained in $E_1$.

(v) By the previous argument, for any divisor $D$ in the interior of the cone $\langle E_1,E_3, P\rangle$, its base locus must be contained in $E_1\cup E_3$.

(vi-vii) Follows easily from what we said above.

(viii) Any divisor $D$ in the interior of the cone generated by $H_1,H_3$ and $E_2$ the base locus of $D$ must be contained in $E_2$. However, since we know the $\mbox{nef}$ cone, then the base locus of any divisor in the complement of $\mbox{nef}$ cone must be contained in $E_2$. This completes the proof.
\end{proof}

\begin{defn}\label{MOVABLE}
Let $Y$ be a smooth  projective variety over $\CC$. The movable cone $\overline{\mathrm{Mov}}(Y)\subset N^1(Y)$ is the closure of the cone generated by classes of effective Cartier divisors $L$ such that the base locus of $|L|$ has codimension at least two. We say that a divisor is movable if its numerical class lies in $\overline{\mathrm{Mov}}(Y)$.
\end{defn}

\begin{cor} The movable cone $\mathrm{Mov}(X_3)$ of $X_3$ is the closed cone spanned by non-negative linear combinations of $H_1,H_2,H_3$ and $P$.
\end{cor}

\medskip\noindent

\section{Birational Models of complete quadric surfaces}\label{CUATRO}
In this section we describe some birational models of the space $X_3$. We present the results very explicitly at the risk of making proofs longer than optimal. This approach will exhibit the moduli structure on the birational models constructed.

\subsection*{Second Order Chow Variety $\mathbf{Chow}_2(1, X_3)$} 
We define the second order Chow variety, $\mathbf{Chow}_2(1, X_3)$, as the parameter space of tangent lines to complete quadric surfaces. More precisely, let $Q\in X_3$ be a smooth complete quadric and let $TQ$ denote the set of tangent lines to it in the Grassmannian $\GG=\GG(1,3)$. Since the class $[TQ]=2\sigma_1\in A^1(\GG)$, it follows that the subvariety $TQ$ is defined by an element in the linear system $|\mathcal{O}_{\GG}(2)|$. Consequently, we have a map $Q\mapsto TQ\in \PP H^0(\GG,\mathcal{O}(2))\cong \PP^{19}$. The subvariety $TQ$ is the so-called quadric line-complex.

\begin{lemma}
Let $X_3^{\circ}\subset X_3$ be the open subset parameterizing smooth quadric surfaces. Then, we have an embedding $$\phi: X_3^{\circ}\rightarrow \PP(H^0(\GG,\mathcal{O}(2))\cong \PP^{19}$$
by mapping a smooth quadric $Q\mapsto TQ$ to its associated degree $2$ hypersurface $TQ\subset \GG(1,3)$.
\end{lemma}
\begin{proof}
Let $Q$ and $Q'$ two distinct smooth quadrics. Then there exists a point $x\in Q$ which is not in $Q'$. The tangent space $T_xQ$ contains a $1$-parameter family of lines tangent to $Q$ among which only $2$ are also tangent to $Q'$. This says that $TQ\ne TQ'$ as desired.
\end{proof}

\begin{prop}\label{propchow1}
The map $\phi$ extends to a morphism $\rho_2:X_3\rightarrow \mathbf{Chow}_2(1,X_3)$.
\end{prop}
\begin{proof}
By Serre's criteria \cite{EISEN}, the rational map $\phi$ extends to a complement of a subset of codimension $2$ in $X_3$. Furthermore, the space $X_3$ is stratified by $SL_4$-orbits  as follows: there is an open dense orbit $X_3^{\circ}$, codimension $1$ and $2$ orbits $E_i^{\circ}$ and $E_i^{\circ}\cap E_j^{\circ}$ ($i\ne j$) respectively, and a unique closed orbit $E_1\cap E_2\cap E_3$.
Therefore, the result follows if the map $\phi$ extends to each of the $E_i$'s, $i.e.$, $\phi(E_i)$ is well-defined for $i=1,2,3$.

Let us show that the map $\phi:(X_3^{\circ})\rightarrow \PP^{19}$ extends to the divisor $E_1$ by performing the explicit computation. First, we exhibit the extension of the map $\phi$ to the open $SL_4 \CC$-orbit $E_1^{\circ}$.

To simplify the computations, let us assume $Q_t\subset \PP^3$ is the family $Q_t=\{x^2+t(ay^2+byz+cyw+dz^2+ezw+f^2w^2)=0\}$. The limit as $t\rightarrow 0$, is a complete quadric $(Q_0,q[y:z:w])\in E_1$. 

We proceed by definition in order to compute the Chow form $TQ_t$. A line in $l\subset \PP^3$ is the image of the morphism 
$$[\alpha,\beta]\overset{g}{\longmapsto} [a_1\alpha+b_1\beta:\cdots :a_4\alpha+b_4\beta]\in \PP^3\ .$$
The line $l$ is tangent to a quadric $Q$ as long as the the restriction of $Q$ to $l$ consists of a single point (with multiplicity two). Therefore, the discriminant $B^2-4AC=0$ of the quadratic polynomial in $[\alpha,\beta]$,

\begin{equation*}
\begin{aligned}
g^*Q_t&= A\alpha^2+B\alpha\beta+C\beta^2\\
&=(a_1^2+t(aa_2^2+ba_2a_3+\cdots +fa^2_4))\alpha^2+\\&+(2aa_2b_2+t((a_2b_3+a_3b_2)b+\cdots )\alpha\beta+(b_1^2+t(ab_2^2+\cdots))\beta^2
\end{aligned}
\end{equation*}
describes the equations desired, which in Pl\"{u}cker coordinates is
\begin{equation}\label{chow1}
\begin{aligned}
TQ_t=\{ap_0^2+bp_0p_1+dp_1^2+cp_0p_2+ep_1p_2+fp_2^2+t(\mbox{extra terms})=0\}
\end{aligned}
\end{equation}
This shows that $\phi:X_3^{\circ}\rightarrow \PP^{19}$ can be extended to the whole of $E_1$. Similar computations show that there are extensions to all of $E_2$ and $E_3$. Indeed, since $E_2$ is $SL_4\CC$-invariant, then we can assume that $Q_t=xy+t(az^2+bzw+cw^2)$ and analyze the normal directions at this point. Following the same argument as above, we find that the associated hypersurface, in Pl\"{u}cker coordinates is $$\Sigma_1(Q_t)=\{p_0^2+t(\mbox{other terms})=0\}\ .$$
It follows that the (radical of the) limit as $t\rightarrow 0$ coincides with the Schubert cycle $\Sigma_1(L)\subset \GG(1,3)$ where $L=\{x=y=0\}$. This gives the natural extension for $\phi_{|E_2}$ as desired. The case for $E_3$ is clear. This completes the proof.
\end{proof}
Semple's notation for $\phi(X_3)$ is $C_9^{92}[19]$. He showed \cite{SEMII} that $\rho_2(X_3)$ is normal. Hence, the following result follows from Theorem B.
\begin{cor}\label{Chow}
The morphism $\rho_2:X_3\rightarrow \mathbf{Chow}_2(1,X_3)\subset \PP^{19}$ contracts the divisor $E_2$. Furthermore, $X_3(H_2)\cong \mathbf{Chow}_2(1,X_3)$.
\end{cor}

\begin{rmk}
The degree of $\mathbf{Chow}_2(1,X_3)\subset \PP^{19}$ is $92$, and has the following enumerative significance. It is the number of smooth quadric hypersurfaces in $\PP^3$ which are tangent to $9$ fixed lines in general position. 
\end{rmk}

\medskip
\subsection*{The flip of $X_3$}
We now construct the flip $X^{+}_3$ of the space of complete quadric surfaces. We do so by analyzing a $\ZZ/2$-action on $\mathbf{Hilb}^{2x+1}(\GG(1,3))$.

\begin{defn}
Let $\mathbf{Hilb}=\mathbf{Hilb}^{2x+1}(\mathbb{G}(1,3))$ denote the Hilbert scheme parametrizing subschemes of $\mathbb{G}(1,3)\subset \PP^5$ whose Hilbert polynomial is $P(x)=2x+1$.
\end{defn}

\begin{prop}
Let $\mathbf{Hilb}$ be as defined above, then $$\mathbf{Hilb}\cong \mbox{Bl}_{\mathbb{OG}}(\mathbb{G}(2,5))$$ where $\mathbb{O G}\subset \GG(2,5)$ denotes the Orthogonal Grassmannian inside the Grassmannian of $2$-planes in $\PP^5$.
\end{prop}
\begin{proof}
Observe that a generic smooth curve with Hilbert polynomial $P(x)=2x+1$ in $\PP^5$ is a plane conic $C$. Thus, its ideal $I_C\subset k[p_0,...,p_5]$ is generated by a quadric $F$ and three independent linear forms $L_1,L_2,L_3$. Since $C\subset \mathbb{G}=\mathbb{G}(1,3)$, the equation $F$ is the quadric corresponding to the Grassmannian $\GG \subset \PP^5$ under the Pl\"{u}cker embedding.
This description gives rise to a rational map $$f:\GG(2,5)\dashrightarrow \mathbf{Hilb} $$ by assigning the $2$-plane $P$ defined by the independent linear forms $(L_1,L_2,L_3)$ to the ideal $\langle L_1,L_2,L_3\rangle+\langle F\rangle \subset k[p_0,...,p_5]$. 
Observe the exceptional locus of $f$ consists of planes in $\PP^5$ such that there is a containment of ideals $\langle F\rangle \subset \langle L_1,L_2,L_3\rangle$ $i.e.$, planes $P$ which are contained in the quadric $\GG \subset \PP^5$. We denote the locus parametrizing such planes by $\mathbb{OG}$. This locus is precisely the Orthogonal Grassmannian. Now, we resolve the rational map $f$,
$$
\xymatrix @!=2pc{ **[c]Bl_{\mathbb{ O G}}(\GG(2,5))&& 
\\
**[c]\GG(2,5)\ar@{<-}[u]^{\pi} \ar@{-->}[rr]_{f} &  &
\mathbf{Hilb} \ar@{<-}[ull]_{\tilde f}
 }
$$ 
The morphism $\tilde f$ is an isomorphism. Indeed, the rational map $f$ is birational as it has an inverse morphism $g:\mathbf{Hilb}\rightarrow \GG(2,5)$ defined as follows: let $[C]\in \mathbf{Hilb}$ be a generic element, then the ideal $I(C)=(F)+ (\mbox{plane})\overset{g}{\mapsto}(\mbox{plane})\in \GG(2,5)$. It is clear that $f \circ g=Id$, hence $f$ and consequently $\tilde f$ are birational. Furthermore, $\tilde f$ is a bijection. Indeed, since the exceptional divisor $E\subset Bl_{\mathbb{OG}}(\GG(2,5))$ is a $\PP^5$-bundle over $\mathbb{OG}$, then we can write $p=(P,C)$ where $P\subset \PP^5$ is a $2$-plane and $C\subset P$ is a plane conic. Thus,  Zariski's Main Theorem implies that $\tilde f$ is an isomorphism.
\end{proof}
\begin{cor}
Let $\mathbf{Hilb}$ be as above, then $\mathrm{Pic}(\mathbf{Hilb})\cong \langle H^{+},E_2^{+},E_{1,1}^{+}\rangle$ where $H^{+}$ is the pullback of $\sigma_1\in A^1(\GG(2,5))$ and the $E^{+}$'s are the exceptional divisors of the blowup.
\end{cor}
\begin{proof}
The orthogonal Grassmannian $\mathbb{OG}$ has two components, hence the result follows.
\end{proof}

If the field $k$ is algebraically closed, then for a given smooth quadric $Q\subset \PP^3_k$, the Fano variety of lines $F_1(Q)\subset \GG(1,3)$ consists of two smooth conics. By exchanging such conics we get a $\ZZ/2$-action on $\mathbf{Hilb}^{2x+1}(\GG(1,3))$. 

\medskip\noindent
\begin{lemma}
There is a nontrivial globally defined $\ZZ/2$-action on $\mathbf{Hilb}^{2x+1}(\GG(1,3))$. 
\end{lemma}
\begin{proof} Let $Q\subset \PP^3$ be a smooth quadric hypersurface. The Fano variety of lines $F_1(Q)$ is the zero locus of a section of the following bundle, \begin{equation*}
\xymatrix @!=3pc{ **[c] Sym^2(S^*)\ar[d]^{\pi} \\
\mathbb{G}(1,3) \ar@/^2pc/[u]^{Q|L}.}
\end{equation*}
A smooth conic in $\PP^5$ determines uniquely a $2$-plane, thus in the Pl\"{u}cker embedding $\GG(1,3)\subset \PP^5$, we have that
\begin{itemize}
\item[(1)] $F_1(Q)$ determines two $2$-planes if $rank(Q)$ is either $2$ or $4$,
\item[(2)] $F_1(Q)$ determines a single $2$-plane if $rank(Q)$ is either $1$ or $3$.
\end{itemize}
Exchanging such planes gives rise to a $\ZZ/2$-action on 
$\GG(2,5)$, the Grassmannian of $2$-planes in $\PP^5$. Clause $(2)$ above, says that such a 
$\ZZ/2$-action on $\GG(2,5)$ preserves the Orthogonal Grassmannian $\mathbb{OG}$, hence inducing a $\ZZ/2$-action on the blowup $\mathbf{Hilb}^{2x+1}(\GG(1,3))$.
\end{proof}

Observe that there is a $SL_4\CC$-action on $\mathbf{Hilb}$ induced by the action of $SL_4$ on $\PP^3$. This action stratifies $\mathbf{Hilb}$ in $SL_4$-orbits compatible with the exceptional divisors $E_2^+,E_{1,1}^+$.
Notice that $\ZZ/2$ acts trivially (as the identity) over $SL_4$-orbits of codimension $2$. In codimension $1$, we have that $\ZZ/2$ acts as the identity on the exceptional divisors $E_2^{+}$ and $E_{1,1}^{+}$. 

\medskip
\begin{defn} Considering the $\ZZ/2$-action defined above, 
let us denote the quotient  $X^{+}_3:=\mathbf{Hilb}/(\ZZ/2)$.
\end{defn}

\medskip
Let $\overline{\mathcal{M}}_{0,0}(\GG,2)$ be the Kontsevich moduli space of degree $2$ stable maps into the Grassmannian $\GG=\GG(1,3)$.
\begin{lemma}\label{2COVER}
There is a nontrivial globally defined $\ZZ/2$-action on the Kontsevich moduli space $\overline{\mathcal{M}}_{0,0}(\GG(1,3),2)$. 
\end{lemma}
\begin{proof}
Observe we have a generic 2-1 morphism from the Kontsevich moduli space $\overline{\mathcal{M}}=\overline{\mathcal{M}}_{0,0}(\GG(1,3),2)=\{(C,C^* \}$ to the space $X_3$ of complete quadric surfaces defined as follows 
$$(C,C^*)\mapsto \left(\bigcup_{L\in C}L,C^*\right)$$
where the notation $(S,C^*)$ means a surface $S$, and a curve $C^*$ as its marking. This map is $2$ to $1$ over the open subset parametrizing smooth quadric surfaces as well as over the divisor of complete quadrics of rank $2$. Indeed, if $\bigcup_{L\in C} L$ sweeps out a smooth quadric $S$, then $L$ is a ruling of $S$. The other ruling induces another conic $C'$ which gets mapped to $S$. The situation is similar over the locus of complete quadrics of rank 2. Notice that this map is 1-1 outside two such regions. We now define the $\ZZ/2$-action on $\overline{\mathcal{M}}$ by identifying the two curves $C$ and $C'$.
\end{proof}

\begin{cor}
The quotient of $\overline{\mathcal{M}}_{0,0}(\GG,2)$ by the $\ZZ/2$-action is isomorphic to $X_3$. In particular, the quotient is smooth.
\end{cor}
\begin{proof}
Let $Z$ denote the quotient of $\overline{\mathcal{M}}$ by the $\ZZ/2$-action defined above. Observe that $X_3$ and $Z$ are birational and there is a bijection between them. Zariski's Main Theorem implies the corollary.
\end{proof} 

\section{Proof of Theorem A}~\label{CINCO}

\noindent
Corollary \ref{MOVABLE} claims that the movable cone $\mathrm{Mov}(X_3)$ is the closed cone spanned by non-negative linear combinations of $H_1,H_2,H_3$ and $P$, where the latter divisor is defined in Definition \ref{P}. In this section we describe all the models $X_3(D)$, where $D\in \mathrm{Mov}(X_3)$. We interpret the spaces constructed thus far, $\mathbf{Chow}_2(1,X_3)$ and $X_3^+$, as models $X_3(D)$ for some $D$ in $\mathrm{Mov}(X_3)$, which exhibits the moduli structure on the models.

\begin{prop}\label{CC}
There is a morphism from $X^{+}_3=\mathbf{Hilb}/(\ZZ/2)$ to the $\ZZ/2$-Chow variety defined by forgetting the scheme structure and only considering its cycle class. Likewise, there is a morphism from the space of complete quadrics $X_3$ to the same $\ZZ/2$-Chow variety.
\end{prop}
\begin{proof}
Define the following spaces: let $I=\{(C,C^*,\Lambda)\}$ be the incidence correspondence such that $C$ is a connected, arithmetic genus zero, degree two curve in $\GG(1,3)\cap \Lambda$ and $C^*$ is the dual curve in $\Lambda^*$. Let $\overline{\mathcal{M}}_{0,0}(\GG,2)$ be the Kontsevich space of degree two stable maps into the Grassmannian $\GG=\GG(1,3)$. Let $\mathcal{C}$ denote the Chow variety of conics in $\PP^5$ which are contained in $\GG(1,3)$. The incidence correspondence $I$ admits a map to both $\overline{\mathcal{M}}_{0,0}(\GG,2)$ and $\mathbf{Hilb}$ by projecting to the first two factors, and by projecting to the first and third factors, respectively. By projection to the first factor, we get a map to $\mathcal{C}$. Since the morphisms $Kh$ and $Ch$ are small contractions, and $\ZZ/2$ acts trivially in $SL_4$-orbits of codimension $2$ and higher, then the Chow variety $\mathcal{C}$ inherits a $\ZZ/2$-action. We thus have the following,
\begin{equation}\label{DIAGRAM}
\xymatrix @!=2pc{ & I  \ar@{->}[dl]\ar@{->}[dr] & \\
 **[c] \overline{\mathcal{M}}_{0,0}(\GG,2) \ar@{->}[dr]^{Kh}\ar@{->}[d]_{\ZZ/2} &  &\ar@{->}[dl]_{Ch} \mathbf{Hilb} \ar@{->}[d]^{\ZZ/2}\\
**[c] X_3 \ar@{->}[dr]_{\pi_1} & \mathcal{C}  \ar@{->}[d] & X^{+}_3 \ar@{->}[dl]^{\pi_2}\\
**[c]  & \mathcal{C}/(\ZZ/2) &  
 }
\end{equation}
The existence of the morphisms $\pi_1$ and $\pi_2$ follows from the fact that $X_3$ as well as $X^{+}_3$ are $\ZZ/2$-quotients.
\end{proof}

We can identify models $X_3(D)$ thanks to the following. 

\begin{lemma}
Let $f:X\rightarrow Y$ be a birational morphism, where $X$ and $Y$ are normal projective algebraic varieties. Let $D\subset Y$ be an ample divisor, then $$\mathrm{Proj}(\oplus_{m\ge 0}H^0(X,f^*D))=Y.$$
\end{lemma}

In the main Theorem of this section, we list only the models $X_3(D)$ for which we have exhibited a moduli interpretation.

\begin{thm1}\label{T1}
Let $D$ be an integral effective divisor in the space of complete quadric surfaces $X_3$, and let $$X_3(D)=\mathrm{Proj}\left( \bigoplus_{m\ge 0}H^0(X_3,  mD)\right)$$
be the model induced by $D$. Then, we have the following models for $X_3(D)$,
\begin{itemize}
\item[1.] $X_3(D)\cong X_3$ for $D$ contained in the cone spanned by $H_1$, $H_2$ and $H_3$.
\item[2.] $X_3(H_1)\cong \mathbf{Hilb}^{(x+1)^2}(\PP^3)\cong \PP^9$ and $f:X_3\rightarrow X_3(H_1)$ contracts the divisors $E_1$ and $E_2$.
\item[3.]$ X_3(H_3)\cong \mathbf{Hilb}^{(x+1)^2}(\PP^{3*})\cong \PP^{9*}$ and $g:X_3\rightarrow X_3(H_3)$ contracts the divisors $E_3$ and $E_2$.
\item[4.]$ X_3(H_2)\cong  \mathbf{Chow}_2(1,X_3)$ and $\phi:X_3\rightarrow X_3(H_2)$ contracts the divisor $E_2$.
\item[5.] $X_3(D)\cong \mathcal{C}/(\ZZ/2)$ where $\mathcal{C}$ is the Chow variety of Proposition $21$ and $D=tH_1+(1-t)H_3$ for $0<t<1$. The map $\phi_1:X_3\rightarrow \mathcal{C}/(\ZZ/2)$ is a small contraction, whose exceptional locus is $Exc(\phi_1)=E_1\cap E_3$.
\item [6.] $X_3(D)\cong X^{+}_3$ for $D$ contained in the domain spanned by $H_1$, $H_3$ and $P$. The map $\phi_2: X^{+}_3\rightarrow \mathcal{C}/(\ZZ/2)$ is the flip of $\phi_1$, where the flipping locus consists of subschemes supported on a line.
\item [7.] $X_3(P)\cong \GG(2,5)/(\ZZ/2)$, where $P$ is defined in Definition \ref{P}.
\end{itemize}
\end{thm1}

The result above can be best seen from the following diagram:
\begin{equation}\label{DIAGRAM}
\xymatrix @!=3pc{ \PP^9 & \ar[l]_{2} X_3\ar[dr]_{\phi_1}^5\ar[d]_4 \ar[dl]_3 \ar@{.>}[rr]_{\mbox{flip}}^6&    &X^+_3 \ar[dr]^{7} \ar[dl]^{\phi_2} & 
\\
 \PP^{9*}& \mathbf{Chow}_2 &  \mathcal{C}/(\ZZ/2) & & \GG(2,5)/(\ZZ/2) \ . } 
\end{equation}
where $\phi_1$ and $\phi_2$ are small contractions and the other maps, except for the flip, are all divisorial contractions. Observe that from Corollary 16 and Proposition 7 we know abstractly all the divisorial contractions of $X_3$ or $X^+_3$ induced by $\mathrm{Mov}(X_3)$. 

\begin{proof}[Proof of Theorem A]
(1), (2), (3) follow from Theorem B and the description of $X_3$ given in section $2$. 

(4) This is established in corollary \ref{Chow}.

(5) Let $C_{1,2}$ be the curve defined before Proposition \ref{prop2}, whose deformations cover the codimension 2 subvariety $E_1\cap E_3$. Observe that for any divisor $D=t H_1+(1-t)H_3$ where $0<t<1$, we have that $C_{1,2}.D=0$,
which says the map $\phi_D$ contracts the codimension 2 locus $E_1\cap E_3$. 
The locus which is contracted does not get any larger as the map $X_3\rightarrow \mathcal{C}/(\ZZ/2)$ behaves locally similar to that of diagram (\ref{DIAGRAM}) and by the observation made about the $\ZZ/2$-action on $SL_4$-orbits, its exceptional locus behaves as in \cite{CC}.

(6) By definition, the morphism $\phi_2:X^{+}_3\rightarrow \mathcal{C}/(\ZZ/2)$ is the flip of $\phi_1:X_3\rightarrow \mathcal{C}/(\ZZ/2)$, if for any divisor $D$ in the domain spanned by $H_1$, $H_3$ and $P$, then both $-D$ is $\phi_1$-ample and $D$ is $\phi_2$-ample.

It is important to notice that the $\ZZ/2$-action on $X_3^{+}$ is the identity over the locus which is flipped.

The following analysis takes place in codimension $2$; by the previous observation we can neglect the $\ZZ/2$-action altogether.
We verify that any $D$ as above is $\phi_2$-ample. Note that $\phi_2^{-1}(p)\cong \PP^1$. Indeed, for $L\subset \GG(1,3)$, where $L$ denotes a line, consider the tangent space $\mathbb{T}_L\GG(1,3)\cong \PP^3$. Now $\mathbb{T}_L\GG(1,3)\cap \GG(1,3)$ is a quadric of rank $1$ (a double plane) with a double line $2L$ on it. The pencil of planes containing $L$ are different points of $\mathbf{Hilb}$, however they all map to the same point $[2L]$ of the Chow variety $\mathcal{C}$. This means that the fiber of $\phi_2$ over the point $[2L]$ is a pencil of planes, hence $\PP^1$. Now let  $D=a\overline{H}_1+b\overline{H}_3+cP$ for positive $a,b,c \in \QQ$ and where $\overline{H}_1$ and $\overline{H}_3$ are defined in $\mathrm{Pic}(\mathbf{Hilb})$ as follows.
$\overline{H}_1=\{(C,\Lambda)|\ C\cap \sigma_2(Pl)\ne \emptyset \}$ for a fixed plane $Pl\subset \PP^3$, and $\overline{H}_3=\{(C,\Lambda)|C\cap \sigma_{1,1}(p)\ne \emptyset \}$ for a fixed point $p\in \PP^3$. Consequently, for the curve $\gamma=\phi^{-1}_2(2L)$, we have
\begin{equation*}
\begin{aligned}
\gamma.D&=\gamma.(a\overline{H}_1+b\overline{H}_3+cP)\\
&=c(\gamma.P) \\
&=c(\gamma.\sigma_1) \quad \mbox{in }\GG(2,5)\\
&> 0 \ .
\end{aligned}
\end{equation*}
Thus, $D$ is $\phi_2$-ample.

Let us now describe the fiber $\phi^{-1}_1(p)$. By Nakai's ampleness criteria, $-D$ will be $\phi_1$-ample if and only if $D.\gamma<0$ for any curve $\gamma$ contained in the fiber of $\phi_1$. The curve $C_{1,2}$ is contained in such a fiber as it is contracted by $\phi_1$. Hence, by Lemma \ref{DIVISORP},
\begin{equation*}
\begin{aligned}
C_{1,2}.D&=C_{1,2}.(aH_1+bH_3+cP)\\
&=c(C_{1,2}.P) \\
&=-c(C_{1,2}.H_2)\\
&< 0 \ .
\end{aligned}
\end{equation*}
Thus, $-D$ is $\phi_1$-ample.

(7) Follows by construction.
This completes the proof of Theorem A.
\end{proof}

\medskip
\section{Historical Remarks and Future Work} \label{SEIS}
\noindent
In this section, we state the historically accurate definition of the space $X_n$, and comment on the relation of Semple's work \cite{SEM}, \cite{SEMII} to the results in this paper. Moreover, we link our work to GIT and representation theory.

\medskip 
Let us recall the historically accurate definition of the space of complete quadrics. Under this definition, the description of $X_n$ in Definition \ref{1.2} is a Theorem in \cite{VAIN}.
Let $Q\subset \PP^{n}=\PP(V)$ be a smooth quadric hypersurface. It defines a symmetric linear map $Q:V\rightarrow V^*$, which induces a symmetric linear map $\wedge ^k Q: \wedge^k V\rightarrow \wedge^k V^*$ for any $1\le k\le n$.  Hence, $\wedge^kQ\in S^2(\wedge^kV)$. If $Q$ is smooth, then the association $Q\mapsto \wedge^kQ$ is injective up to multiplication by scalars.
Consequently, we get an embedding of $X^{\circ}_n$, the family of smooth quadric hypersurfaces in $\PP(V)$, into the space $W=\PP(S^2 (V))\times\PP(S^2 (\wedge^2V))\times\ldots \times \PP(S^2(\wedge^{n}V))$ through the map
$$\rho: Q\mapsto (Q,\wedge^2Q,\ldots ,\wedge^{n}Q) .$$

\begin{defn} \label{CQUADRICS} The space of complete $(n-1)$-quadrics $X_n$ is defined as the closure $\overline{\rho(X_n^{\circ})}\subset W$. 
\end{defn}

\medskip
Consider $X_3\subset \PP^9\times \PP^{19}\times \PP^{9*}=\PP_1\times \PP_2\times \PP_3$, the space of complete quadric surfaces. Theorem A claims the image of the projection map $\rho_i:X_3\rightarrow \PP_i$ is isomorphic to $X_3(H_i)$, for $1\le i\le 3$. One can also consider the projection to $$\rho_{i,j}:X_3\rightarrow \PP_i\times \PP_j\ ,$$ for $1\le i<j\le 3$. Semple focused on the projections $\rho_2$ and $\rho_{1,3}$.
For example, he denotes the space $\rho_2(X_3)$ by $C^{92}_9[19]$ and carefully studies its singularities.
By Proposition \ref{X1}, the projection $\rho_{1,2}(X_3)$ (respectively, $\rho_{2,3}(X_3)$) is a divisorial contraction isomorphic to $Bl_{\Phi_1}\PP^9$ (respectively, $Bl_{\Phi_1}\PP^{*9}$). The projection $\rho_{1,3}(X_3)$ is of a different kind. It is a small contraction which Semple denotes by $W_9$. He carefully analyzes the singularities of $W_9$ as well as its geometry. We have seen, these spaces are the models arising from divisors in the nef cone of $X_3$.

\medskip
Birational models of $X_3$ which are not analyzed in \cite{SEMII} arise when we study the models $X_3(D)$ induced by divisors $D$ which are not nef, but which are contained in the movable cone. One such example is the flip of $X_3$ over $\rho_{1,3}(X_3)$. 

\medskip
On the other hand, the space of complete $(n-1)$-quadrics can be obtained as a GIT quotient. Indeed, De Concini and Procesi in \cite{DECON} constructed the ``wonderful compactification" of a symmetric variety. Viewing $SL_{n+1}(\CC)\cong SL_{n+1}(\CC)\times SL_{n+1}(\CC)/\Delta$, as a symmetric variety, one can consider the wonderful compactification $\overline{G}=\overline{SL_{n+1}\CC}$. This is a $H$-variety, where $H$ is the fixed subgroup of the $SL_{n+1}$-involution $A\mapsto \  ^{t}\!A^{-1}$. Thus, we can take the GIT-quotient $\overline{G}^{ss}/\!/H$ for a suitable choice of linearization of $H$. This quotient is a compactification of $SL_{n+1}/H$, which is isomorphic to complete $(n-1)$-quadrics \cite{KAN}. This point of view suggests that we might understand Theorem A as a variation of GIT. 

\medskip
Observe that the models $X_3(D)$, for $D\in \mathrm{Mov}(X_3)$, are $SL_4$-equivariant compactifications of the homogeneous space $SL_4/\overline{SO}(4)$. Then, the results of this paper admit a description in terms of the Luna-Vust Theory on compactifications of spherical varieties \cite{LUNA}. The latter theory is written in representation-theoretic terms, and aims to understand the $G$-equivariant embeddings of homogeneous spaces $G/H\rightarrow X$, where $X$ is a $G$-variety. In future work, we aim to study the relation among the $SL_4$-equivariant compactifications of $SL_4/\overline{SO}(4)$, \`a la Vust-Luna, and the small modifications (Definition \ref{MOVABLE}) of the De Concini-Procesi wonderful compactification of $SL_4/\overline{SO}(4)$.

\bigskip
\nocite{TAD}
\nocite{CHE}
\nocite{DECON}
\nocite{DGP}
\nocite{HART}
\nocite{LAZ}
\nocite{GH}
\nocite{JANOS}
\nocite{CHJ}
\nocite{HARTD}
\nocite{JANOS3}
\nocite{LAK}
\nocite{KLE}


\begin{thebibliography}{DCGMP88}

\bibitem[ABCH13]{ABCH}
D.~Arcara, A.~Bertram, I.~Coskun, and J.~Huizenga, \emph{The minimal model
  program for the {H}ilbert scheme of points on the plane and {B}ridgeland
  stability}, Advances in Mathematics \textbf{325} (2013), 580--626.

\bibitem[BCHM]{BC}
P.~Birkar, P.~Cascini, C.~Hacon, and J.~McKernan, \emph{Existence of minimal
  models for varieties of log general type}, J. Amer. Math. Soc. (2010), no.~2,
  405--468.

\bibitem[Ber]{BERTRAM}
A.~Bertram, \emph{An application of a log version of the {K}odaira vanishing
  theorem to embedded projective varieties}, arxiv.org/abs/alg-geom/9707001.

\bibitem[CC10]{CC}
D.~Chen and I.~Coskun, \emph{Stable base locus of {K}ontsevich moduli spaces},
  Michigan Mathematical Journal \textbf{59} (2010), no.~2, 435--466.

\bibitem[Cha64]{CHAS}
M.~Chasles, \emph{D\`etermination du nombre de sections coniques qui doivent
  toucher cinq courbes donne\'es d'ordre quelconque, ou satisfaire \`a diverses
  autres conditions}, C.R. Acad. Sci. Paris \textbf{58} (1864), 222--226.

\bibitem[Che05]{CHE}
D.~Chen, \emph{Mori's program for the {K}ontsevich moduli space
  $\overline{M}_{0,0}(\mathbb{P}^3,3)$}, Int. Math. Res. Not. \textbf{235}
  (2005), 17 pp.

\bibitem[CHS05]{CHJ}
I.~Coskun, J.~Harris, and J.~Starr, \emph{Effective cone of the {K}ontsevich
  moduli space}, Int. Math. Res. Not. \textbf{235} (2005), 17 pp.

\bibitem[DCGMP88]{DGP}
C.~De~Concini, M.~Goresky, R.~MacPherson, and C.~Procesi, \emph{On the geometry
  of quadrics and their degenerations}, Comment. Math. Helvetici \textbf{63}
  (1988), 337--413.

\bibitem[DCP80]{DECON}
C.~De~Concini and C.~Procesi, \emph{Complete symmetric varieties}, Invariant
  Theory (1980), Springer-Verlag.

\bibitem[Eis95]{EISEN}
D.~Eisenbud, \emph{Commutative {A}lgebra with a view towards {A}lgebraic
  {G}eometry}, Graduate {T}exts in {M}athematics, Springer-{V}erlag, 1995.

\bibitem[ELMMP1]{POPA1}
L.~Ein, R.~Lazarsfeld, M.~Musta{\c{t}}{\u{a}}, M.~Nakamaye, and M.~Popa,
  \emph{Asymptotic invariants of base loci}, Ann. Inst. Fourier (Grenoble)
  \textbf{56} (2006), no.~6, 1701--1734.

\bibitem[ELMMP2]{POPA2}
\bysame, \emph{Restricted volumes and base loci of linear series}, Amer. J.
  Math. \textbf{131} (2009), no.~3, 607--651.

\bibitem[GH78]{GH}
P.~Griffiths and J.~Harris, \emph{Principles of {A}lgebraic {G}eometry}, Wiley
  {I}nterscience, 1978.

\bibitem[Har78]{HART}
R.~Hartshorne, \emph{Algebraic {G}eometry}, Graduate {T}exts in {M}athematics,
  Springer-{V}erlag, 1978.

\bibitem[Har09]{HARTD}
\bysame, \emph{Deformation {T}heory}, Graduate {T}exts in {M}athematics,
  Springer-{V}erlag, 2009.

\bibitem[Has05]{HAS}
B.~Hassett, \emph{Classical and minimal models of the moduli space of curves of
  genus two}, Geometric methods in algebra and number theory \textbf{235}
  (2005), 169--192, Birkhuser, Boston, MA,.

\bibitem[HH09]{HHD}
B.~Hassett and D.~Hyeon, \emph{Log canonical models for the moduli space of
  curves: {T}he first divisorial contraction}, Transations of the AMS
  \textbf{361} (2009), no.~4, 1065--1080.

\bibitem[HH13]{HHF}
\bysame, \emph{Log minimal models for the moduli space of curves: {T}he first
  flip}, Ann. of Math. \textbf{177} (2013), no.~3, 911--968.

\bibitem[Kan99]{KAN}
S.~Senthamarai Kanna, \emph{Remark on the wonderful compactification of
  semisimple algebraic groups}, Proc. Indian Acad. Sci. \textbf{109} (1999),
  no.~3, 241--256.

\bibitem[KM08]{JANOS}
J.~Koll\'ar and S.~Mori, \emph{Birational geometry of algebraic varieties},
  Cambidge {T}racts in {M}athematics, 2008.

\bibitem[Kol87]{JANOS3}
J.~Koll\'ar, \emph{The structure of algebraic threefolds: an introduction to
  {M}ori's {P}rogram}, Bulletin of the Amer. Math. Soc. \textbf{17} (1987),
  no.~2, 211--273.

\bibitem[KT87]{KLE}
S.~Kleiman and A.~Thorup, \emph{Intersection theory and enumerative geometry: a
  decade in review}, AMS Proc. of Symp. Pure Math (Algebraic geometry:
  {B}owdoin 1985) \textbf{46-2} (1987), 321--370.

\bibitem[Lak87]{LAK}
D.~Laksov, \emph{complete quadrics and linear maps}, AMS Proc. of Symp. Pure
  Math (Algebraic geometry: {B}owdoin 1985) \textbf{46-2} (1987), 371--387.

\bibitem[Laz04]{LAZ}
R.~Lazarsfeld, \emph{Positivity in {A}lgebraic {G}eometry {I}, classical
  setting: line bundles and linear series}, Springer-{V}erlag, 2004.

\bibitem[LV83]{LUNA}
D.~Luna and Th. Vust, \emph{Plongements d'espaces homog\`enes}, Comment. Math.
  Helv. \textbf{58} (1983), no.~2, 186--245.

\bibitem[Sch79]{Schu}
H.~Schubert, \emph{Kalkul der abzahlenden geomretrie}, Leipsig (1879),
  no.~Reprinted by Springer-Verlag, Berlin (1979).

\bibitem[Sem48]{SEM}
J.~G. Semple, \emph{On complete quadrics {I}}, J. Lon. Math. Soc \textbf{23}
  (1948), 258--267.

\bibitem[Sem52]{SEMII}
\bysame, \emph{On complete quadrics {II}}, J. Lon. Math. Soc (1952), 280--287.

\bibitem[Tha01]{TAD}
M.~Thaddeus, \emph{Complete collineations revisited}, arXiv:math/9808114
  (2001).

\bibitem[Tho96]{THROOP}
A.~Thorup, \emph{Parameter spaces for quadrics}, Banach Center Publications
  \textbf{36} (1996).

\bibitem[Tyr56]{TYRELL}
J.A. Tyrell, \emph{Complete quadrics and collineations in {$S_n$}}, Mathematic
  \textbf{3} (1956), 69--79.

\bibitem[Vai82]{VAIN}
I.~Vainsencher, \emph{Schubert calculus for complete quadrics}, Enumerative
  geomety and classical algebraic geometry, Birkh\"{a}user, Boston, Mass, 1982,
  pp.~199--235.
\end{thebibliography}

\end{document}